\documentclass[12pt]{amsproc}

 \usepackage{color}

\newtheorem{theorem}{\sc Theorem}[section]
\newtheorem{lemma}[theorem]{\sc Lemma}
\newtheorem{proposition}[theorem]{\sc Proposition}

\newtheorem{problem}[theorem]{\sc Problem}

 \DeclareMathOperator{\soc}{soc}

 \newcommand{\ol}{\overline}

\newcommand{\la}{\lambda}

\author{Eloisa Detomi}
\address{Dipartimento di Matematica, Universit\`a di Padova, 
 Via Trieste 63, 35121 Padova , Italy}
\email{detomi@math.unipd.it}
\author{Pavel Shumyatsky}
\address{Department of Mathematics, University of Brasilia,
Brasilia-DF, 70910-900 Brazil}
\email{pavel@unb.br}

\title[On the length of a finite group]{On the length of a finite group and of its 2-generator subgroups}
\thanks{The research of the first author was partially supported by GNSAGA.  The research of the second was  supported by CAPES and CNPq.}
 \thanks{2010 {\it Mathematics Subject Classification. }20D10,20F22}
 \thanks{Keywords: finite groups; Fitting height; nonsoluble length}
\begin{document}

\begin{abstract} The nonsoluble length $\lambda(G)$ of a finite group $G$ is defined as the minimum number of nonsoluble factors in a normal series of $G$ each of whose quotients either is soluble or is a direct product of nonabelian simple groups. The generalized Fitting height of a finite group $G$ is the least number $h=h^*(G)$ such that $F^*_h(G)=G$, where  $F^*_1(G)=F^*(G)$ is the generalized Fitting subgroup, and $F^*_{i+1}(G)$ is the inverse image of $F^*(G/F^*_{i}(G))$. In the present paper we prove that if $\lambda (J)\leq k$ for every 2-generator subgroup $J$ of $G$, then $\lambda(G)\leq k$. It is conjectured that if $h^*(J)\leq k$ for every 2-generator subgroup $J$, then $h^*(G)\leq k$. We prove that if $h^*(\langle x,x^g\rangle)\leq k$ for all $x,g\in G$ such that $\langle x,x^g\rangle$ is soluble, then $h^*(G)$ is $k$-bounded.

\end{abstract}
\maketitle

\section{Introduction}

Certain properties of a finite group $G$ can be detected by looking at its 2-generator subgroups. For example, it is well-known that $G$ is nilpotent if and only if every 2-generator subgroup of $G$ is nilpotent. A deep theorem of Thompson says that $G$ is soluble if and only if every 2-generator subgroup of $G$ is soluble \cite{thompson} (see also Flavell \cite{flavell}). A number of recent results reflecting the phenomenon that properties of a finite group are determined by its boundedly generated subgroups can be found in \cite{eu10,lmshu,delu}.
 
In the present paper we deal with groups of given nonsoluble length. Every finite group $G$ has a normal series each of whose quotients either is soluble or is a direct product of nonabelian simple groups. In \cite{junta2} the nonsoluble length of $G$, denoted by $\lambda (G)$, was defined as the minimal number of nonsoluble factors in a series of this kind: if
$$
1=G_0\leq G_1\leq \dots \leq G_{2k+1}=G
$$
is a shortest normal series in which  for $i$  even  the quotient $G_{i+1}/G_{i}$ is soluble (possibly trivial), and for $i$ odd the quotient $G_{i+1}/G_{i}$   is a (non-empty) direct product of nonabelian simple groups, then the nonsoluble length $\lambda (G)$ is equal to $k$.

\begin{theorem}\label{a} Suppose that $\lambda (J)\leq k$ for every 2-generator subgroup $J$ of a finite group $G$. Then $\lambda (G)\leq k$.
\end{theorem}

Recall that the generalized Fitting subgroup $F^*(G)$ of a finite group $G$ is the product of the Fitting subgroup $F(G)$ and all subnormal quasisimple subgroups; here a group is quasisimple if it is perfect and its  quotient by the centre is a nonabelian simple group. Then the \textit{generalized Fitting series} of $G$ is defined starting from  $F^*_0(G)=1$, and then by induction, $F^*_{i+1}(G)$ being  the inverse image of $F^*(G/F^*_{i}(G))$. The least number $h$ such that $F^*_h(G)=G$ is defined as the generalized Fitting height $h^*(G)$ of $G$. Clearly, if $G$ is soluble, then $h^*(G)=h(G)$ is the ordinary Fitting height of $G$. Bounding the generalized Fitting height of a finite group $G$ greatly facilitates using the classification of finite simple groups (and is itself often obtained using the classification). One of such examples is the reduction of the Restricted Burnside Problem to soluble and nilpotent groups in  the Hall--Higman paper \cite{ha-hi}, where the generalized Fitting height was in effect bounded for groups of given exponent (using the classification as a conjecture at the time). A similar example is Wilson's reduction of the problem of local finiteness of periodic profinite groups to pro-$p$ groups in \cite{wil83}.

In view of our Theorem \ref{a} the following problem is natural.

\begin{problem}\label{b} Suppose that $h^*(J)\leq k$ for every 2-generator subgroup $J$ of a finite group $G$. Does it follow that $h^*(G)\leq k$?
\end{problem}

We were not able to answer the above question. On the other hand, the next result seems to be of independent interest.

\begin{theorem}\label{c} Let $G$ be a finite group in which $h^*(\langle x,x^g\rangle)\leq k$ for all $x,g\in G$ such that $\langle x,x^g\rangle$ is soluble. Then $h^*(G)$ is $k$-bounded.
\end{theorem}

We use the expression ``$k$-bounded" to mean ``bounded from above by a number depending on $k$ only". Our method of proof of Theorem \ref{c} shows that $h^*(G)\leq(k+1)2^{k}-1$. However we do not think that the bound $(k+1)2^{k}-1$ is anywhere near to being sharp. We conjecture that actually, under the hypothesis of Theorem \ref{c}, we necessarily have $h^*(G)\leq k$.

 Thus, the next question is natural.

\begin{problem}\label{d} Does any finite group $G$ contain a soluble subgroup $J$ such that  $h^*(G)=h(J)$?
\end{problem}

A positive answer to Problem \ref{d} would imply a positive answer to Problem \ref{b}.

The proofs in this paper use the classification of finite simple groups in 
its application to Schreier's Conjecture that the outer automorphism groups of finite simple groups are soluble.

\section{Proof of Theorem \ref{a}}

In what follows we denote by $R(K)$ the soluble radical and by $\soc(K)$ the socle of a group $K$. Recall that $\soc(K)$ is the product of all minimal normal subgroups of $K$. If $G$ is a nonsoluble finite group such that $R(G)=1$, then of course $\lambda(G/\soc(G))=\lambda(G)-1$.

\begin{lemma}\label{lem:olG}
Let $G$ be a group with $\la(G)=1$. Let $N=S_1\times\cdots\times S_t$ be a normal subgroup in $G$ which is a direct product of isomorphic nonabelian simple groups. Then the group of permutations induced by $G$ on the set $\{S_1,\ldots,S_t\}$ is soluble. 
\end{lemma}
\begin{proof} 
Let $\ol G$ be the group of permutations induced by $G$ on the set $\{S_1, \ldots,S_t\}$. Since the soluble radical $R(G)$ of $G$ centralises  $N$, and $\la(G)=\la(G/R)$, we can assume that $R(G)=1$. Then, as $\la(G)=1$, it follows that $G/\soc(G)$ is soluble. Since $\soc(G)$ normalises all $S_i$, it follows that $\ol G$ is an homomorphic image of $G/\soc(G)$. Hence $\ol G$ is soluble. 
\end{proof}

\begin{proposition}\label{la=1}
Let $N,M$ be normal subgroups of $G$ such that $\lambda(G/N)\leq\lambda(G/M)\leq1$. Then $\lambda(G/N\cap M)\leq1$. 
\end{proposition}
\begin{proof} Suppose that $G$ is a counterexample of minimal order. Then  $M\cap N=1$. 
Since $G$ is a counterexample of minimal order, it follows that $\lambda(G/L)\leq1$ for any nontrivial 
  normal subgroup of $G$. In particular, the soluble radical $R(G)$ of $G$ must be trivial because $\lambda(G)=\lambda(G/R(G))$. Without loss of generality we can assume that $M$ and $N$ are minimal normal subgroups in $G$. Let $G'$ be the derived subgroup of $G$. Since $\lambda(G')=\lambda(G)$, because of minimality of $|G|$ we have $G'=G$ 
  and $\lambda(G/N)=\lambda(G/M)=1$.
 
Let $N=S_1\times\cdots\times S_t$, where $S_i$ are isomorphic nonabelian simple groups. Since $M$ centralises $N$ and $\lambda(G/M)=1$, Lemma \ref{lem:olG} tells us that the permutation  group $\ol G$ induced by $G$ on 
 $\{S_1, \ldots , S_t\}$ is soluble. Taking into account that $G=G'$, we deduce that $\ol G=1$. 
 
Therefore $t=1$ and $N$ is simple. From the Schreier Conjecture combined with the fact that $G=G'$ we now deduce that $G/C_G(N)$ acts on $N$ by inner automorphisms. Hence, $G=N\times C_G(N)$. Since $\lambda(C_G(N))=\la(N)=1$, the result follows.
 \end{proof} 

Given a finite group $G$, we define $T(G)$ as the intersection of all normal subgroups $N$ of $G$ such that $\lambda(G/N)\leq1$. It is easy to deduce from   Proposition \ref{la=1} that $\la(G/T(G))\leq1$ and $\la(G/T(G))=1$ if and only if $G$ is nonsoluble. 
   
Set $T_1(G)=G$ and, by  induction, $T_{i+1}(G)= T(T_{i}(G))$. In view of Proposition \ref{la=1}, it is clear that
 if $T_{i-1}(G) \neq 1$, then  $ T_i(G)$ is the minimal normal subgroup $N$ of $G$ such that $\lambda(G/N)=i-1$. Moreover, $\lambda(T_{i}(G)/T_{i+1}(G))=1$ and $T_{i}(G)$ is perfect   for every $i\geq1$ such that $T_i(G) \neq 1$.

\begin{lemma}\label{zero}
Let $G$ be a finite group with nonsoluble length $\lambda$. Then for every positive integer  $n\le\lambda$ there exists a subgroup $H$ in $G$ such that $\lambda(H)=n$. 
\end{lemma}
\begin{proof} For example, the subgroup $T_{\la-n+1}(G)$ has the required property.
\end{proof}

\begin{lemma}\label{lem:Centralizer} Let $N$ be a  normal subgroup of $G$. If $N$ is a direct product of nonabelian simple groups and 
 $\lambda(G/N)=\lambda(G)$, then $C_G(N)\neq1$. 
\end{lemma}
\begin{proof} Let $\lambda=\lambda(G)$ and $T=T_\lambda(G)$. Thus, $\lambda(G/T)=\la-1$. Since $\lambda(G/N)=\lambda$, it is clear that  $T$ can not be a subgroup of $N$. If the soluble radical $R(T)$ is nontrivial, we have $R(T)\leq C_G(N)\neq1$. 
Otherwise, $T$, being perfect, is a direct product of nonabelian simple groups. In particular $T$ is a product of the minimal normal subgroups of $G$   contained in $T.$ 
 If $T_0\le T$ is a minimal normal subgroup of $G$, then either  $T_0$ 
 centralizes $N$ or $T_0=[T_0,N]\le N$. Since $T$ is not contained in $N$, we deduce that   $C_G(N)\neq1$. 
\end{proof}

In what follows we let $d(G)$ denote the minimal size of a generating set of a group $G$. We will require the following well-known theorem.
\begin{theorem}\label{L-M} Let the finite group $G$ have a unique minimal normal subgroup $N$ and assume that $G/N$ is noncyclic. Then $d(G)=d(G/N)$.
\end{theorem}
The above theorem was proved in \cite{AG} in the case where $N$ is abelian and in \cite{LM} in the general case.

 We are now ready to prove Theorem \ref{a}.

\begin{proof}[Proof of Theorem \ref{a}] Recall that $\lambda (J)\leq k$ for every 2-generator subgroup $J$ of $G$. Our aim is to show that $\lambda(G)\leq k$. Let $G$ be a counterexample of minimal possible order. Then $\lambda(G)=k+1$. In view of Thompson's theorem \cite{thompson}, $k\geq1$. We deduce from Lemma \ref{zero} that $\lambda(H)\le k$ for every proper subgroup $H<G$. Further, whenever $N$ is a 
 nontrivial normal subgroup, we have $\lambda(G/N)\le k$. Let $T=T_{k+1}(G)$. It follows that $T$ is contained in each 
 nontrivial 
normal subgroup of $G$. Therefore $T$ is a unique minimal normal subgroup in $G$. Since $T=T'$, we conclude that $T$ is a direct product of isomorphic nonabelian simple groups.

By minimality of $G$,  the quotient $G/T$ has a 2-generator subgroup $H/T$ such that $\la(H/T)=k$. If $H=G$, then, by Theorem \ref{L-M}, the group $G$ is 2-generator and we have a contradiction. Assume that $H\neq G$. Since $H$ is a proper subgroup, we have $\lambda(H)\le k$. Since the image of $H$ in $G/T$ has nonsoluble length $k$, we conclude that $\la(H)=k$. Hence, $\la(H/T)=\la(H)$. Lemma \ref{lem:Centralizer} now tells us that $C_H(T)\neq1$. It follows that $C_G(T)\neq1$. Since $T$ is a direct product of nonabelian simple groups and since $T$ is a unique minimal normal subgroup in $G$, the centralizer $C_G(T)$ must be trivial. This is a contradiction. The proof is now complete.
\end{proof}

\section{Proof of Theorem \ref{c}}

Given a finite group $G$, we denote by $K(G)$ the intersection of all normal subgroups $N$ such that $F^*(G/N)=G/N$. By \cite[Lemma X.13.3(c)]{hubla} $F^*(G/K(G))=G/K(G)$.

 Define the series
$$K_1(G)=G, \text{ and } K_{i+1}(G)=K(K_i(G))\ for\ i=1,2,\dots.$$
If $h=h^*(G)$, we have the usual inclusions $K_i (G)\leq F^*_{h-i+1}(G)$ for $i=1,2,\dots,h+1$. Moreover, it is clear that $h^*(G)=i+h^*(G/F^*_i(G))$ and $h^*(G)=i+h^*(K_i(G))-1$.

\begin{lemma}\label{bbou} Let $\lambda$ and $h$ be nonnegative integers and $G$ a finite group such that $\lambda(G)=\lambda$ and $h^*(G)=h$. Then $G$ contains a soluble subgroup $B$ such that $h(B)\geq \frac{h+1-2^\lambda}{2^\lambda}$.
\end{lemma}

\begin{proof} If $\lambda=0$, the result is obvious, so we can assume that $\lambda\geq1$ and use induction on $\lambda$. Let $R=R(G)$ and $\bar{G}=G/R$. Choose a Sylow 2-subgroup $T$ in $\soc(\bar{G})$. By the Frattini lemma $\bar{G}=\soc(\bar{G})N_{\bar{G}}(T)$. Let $H$ be the inverse image of $N_{\bar{G}}(T)$ in $G$. We will show that $\lambda(H)=\lambda-1$ and $2h^*(H)\geq h-1$.

 By the Feit-Thompson Theorem \cite{fei-tho}
$\soc(\bar{G})\cap N_{\bar{G}}(T)$ is soluble. Let $S/R=\soc(\bar{G})$. Then $S\cap H$ is soluble. We know that $H/(S\cap H)$ is isomorphic with $G/S$ and that $\lambda(G/S)=\lambda-1$. Hence, $\lambda(H)=\lambda-1$.

We will now prove that $2h^*(H)\geq h-1$. Let $j$ be the minimal number such that $K_j=K_j(G)$ is soluble. We notice that $h(K_j)=h-j+1$. Hence $h^*(H)\geq h-j+1$. Since $K_{j-1}(G)$ is not contained in $R$, it follows that $K_{j-1}(\bar{G})$ is a nontrivial subgroup contained in $\soc(\bar{G})$ while $K_{j-2}(\bar{G})$ is not contained in $\soc(\bar{G})$. Since $\bar{G}=\soc(\bar{G})N_{\bar{G}}(T)$, it follows that $K_{j-2}(H)\neq1$. Hence $h^*(H)\geq j-2$. Combining this with the fact that $h^*(H)\geq h-j+1$, we conclude that $h^*(H)\geq\frac{h-1}{2}$.

By induction, $H$ contains a soluble subgroup $B$ such that $$h(B)\geq \frac{h^*(H)+1-2^{\lambda-1}}{2^{\lambda-1}}.$$ Since $h^*(H)\geq\frac{h-1}{2}$, we have $$h(B)\geq\frac{h^*(H)+1-2^{\lambda-1}}{2^{\lambda-1}}\geq\frac{h+1-2^\lambda}{2^\lambda}.$$ The proof is now complete.
\end{proof}
We will require the following results obtained in \cite{junta2} and \cite{eu10}, respectively.

\begin{theorem}\label{ju} The nonsoluble length $\lambda(G)$ does not exceed the maximum Fitting height of soluble subgroups of a finite group $G$.
\end{theorem}
\begin{theorem}\label{uuu} Every soluble group $G$ has a subgroup $J$ generated by a pair of conjugate elements such that $h(G)=h(J)$.
\end{theorem}
The proof of Theorem \ref{c} is now straightforward.
\begin{proof}[Proof of Theorem \ref{c}] Recall that $k$ is a positive integer and $G$ a finite group in which $h^*(\langle x,x^g\rangle)\leq k$ for all $x,g\in G$ such that $\langle x,x^g\rangle$ is soluble. We wish to show that $h^*(G)$ is $k$-bounded. Let $J\leq G$ be a soluble subgroup of maximal Fitting height. In view of Theorem \ref{uuu} we can choose $J$ in such a way that $J$ is generated by a pair of conjugate elements. Let $t=h(J)$ and $h=h^*(G)$. By Theorem \ref{ju}, we have $\lambda(G)\leq t$. Now Lemma \ref{bbou} shows that $$\frac{h+1-2^t}{2^t}\leq t.$$ From this we deduce that $h\leq(t+1)2^{t}-1$. So in particular $h$ is bounded by a function of $t$. Since $t\leq k$, the theorem follows.
\end{proof}

\bibliographystyle{amsplain}

\end{document}